\documentclass[12pt]{amsart} 
\usepackage{amsfonts}
\usepackage{amsmath,amssymb,amsthm,amsthm}
\usepackage{mathtools}
\usepackage{dsfont}
\usepackage{etoolbox}

\usepackage{tikz}
\usetikzlibrary{calc,positioning}

\usepackage{todonotes}

\usepackage{esint}

\usepackage[margin=1in]{geometry}

\newtheorem{theorem}{Theorem}[section]

\newtheorem{corollary}[theorem]{Corollary}
\newtheorem{definition}[theorem]{Definition}

\newtheorem{lemma}[theorem]{Lemma}

\newtheorem{proposition}[theorem]{Proposition}
\newtheorem{assumption}[theorem]{Assumption}
\newtheorem{remark}[theorem]{Remark}
\numberwithin{equation}{section}

\newcommand{\bR}{\mathbb{R}}

\newcommand\cD{\mathcal{D}}

\newcommand\cN{\mathcal{N}}

\newcommand\dist{\operatorname{dist}}
\newcommand{\p}{\partial}
\newcommand{\epsi}{\varepsilon}

\usepackage[margin=1in]{geometry}

\begin{document}

\title[Mixed problems]{On the $W^{2,p}$ solvability for mixed boundary value problems}
\subjclass[2020]{Primary ?????; Secondary ?????}

\author[R. Khandelwal]{Rohit Khandelwal}
\address[R. Khandelwal]{Faculty of Mathematical Sciences (FMS), South Asian University, New Delhi, India 110068}
\email{rkhandel@sau.int}
\thanks{R. Khandelwal was partially supported by the Faculty research grant SAU/Admin/2026/234 and ANRF PMECRG ANRF/ECRG/2025/001194/PMS}

\author[Z. Li]{Zongyuan Li}
\address[Z. Li]{Department of Mathematics, City University of Hong Kong, 83 Tat Chee Avenue, Kowloon Tong, Hong Kong SAR}
\email{zongyuan.li@cityu.edu.hk}
\thanks{Z. Li was partially supported by the Hong Kong RGC grant ECS 21307225 and the NSFC grant 12501275.}


\begin{abstract}
We establish global $W^{2,p}$ solvability for mixed Dirichlet--Neumann boundary problems. For $C^{1,\alpha}$ domains with a Reifenberg flat interface, we obtain the optimal range $1<p<4/3$, provided that $\alpha>1-1/p$.
For Lipschitz polyhedra with facewise boundary decompositions, we obtain solvability for $p$ close to $1$. In both settings, we establish endpoint $W^{2,1}$ solvability for data in an adapted atomic Hardy space.
Extensions to Lam\'e systems are also discussed.
\end{abstract}

\maketitle

\section{Introduction}
\label{sec-introduction}

Let $\Omega\subset\bR^n$, $n\geq2$, be a bounded and connected domain. 
We study the mixed Dirichlet--Neumann problem
\begin{equation}
 \left\{
 \begin{aligned}
  -\Delta u&=f &&\text{in }\Omega,\\
  u&= g_\cD &&\text{on }\cD,\\
  \partial_\nu u&= g_\cN
  &&\text{on }\cN.
 \end{aligned}
 \right.
 \tag{M}
 \label{prob:inhomo}
\end{equation}
Here, the boundary is decomposed into two disjoint, nonempty, relatively open sets $\cD$ and $\cN$, satisfying $\p\Omega=\overline{\cD}\cup\overline{\cN}$. 
Let $\Gamma := \overline{\cD}\cap\overline{\cN}$ be their interface.
We denote by $\nu$ the unit outward normal on $\p\Omega$, where it is defined.

Our aim is to determine the geometric assumptions and the range of
$p$ for which \eqref{prob:inhomo} is well posed in $W^{2,p}(\Omega)$:
\begin{center}
\textit{given $f\in L^p(\Omega)$ and boundary data in the corresponding trace
spaces, does there exist a unique solution $u\in W^{2,p}(\Omega)$?}
\end{center}

Unlike the pure Dirichlet or Neumann problems, a solution of
\eqref{prob:inhomo} may be singular near $\Gamma$ even when the domain and data are smooth. See Remark \ref{rmk-optimal}. Thus, second-order estimates for mixed problems naturally lead to a regime of low integrability.
In particular, the source term $f$ needs
not belong to the dual of the energy space, so regularity results for energy solutions do not by themselves give a positive answer to our question.

We first establish estimates and solvability for the problem with
homogeneous boundary conditions:
\begin{equation}
 \left\{
 \begin{aligned}
  -\Delta u&=f &&\text{in }\Omega,\\
  u&= 0 &&\text{on }\cD,\\
  \partial_\nu u&= 0 &&\text{on }\cN.
 \end{aligned}
 \right.
 \tag{P}
 \label{prob:poisson}
\end{equation}

Our first theorem gives solvability throughout the \textit{optimal range} of
$p$ on $C^{1,\alpha}$ domains with a sufficiently flat interface.
Put
\[
 \delta_\Gamma(x):=\dist(x,\Gamma).
\]

\begin{theorem}
 \label{thm:poisson-smooth}
 Let $\Omega\subset\bR^n$, $n\geq2$, be a bounded, connected and
 $C^{1,\alpha}$ domain, let $p\in[1,4/3)$, and assume
 $\alpha>1-1/p$. There exists $\theta_0>0$ such that, if the interface
 $\Gamma$ is $\theta_0$-Reifenberg flat, the following assertions hold.
 \begin{enumerate}
  \item If $p=1$ and $f\in H^1_{\cD}(\Omega)$, there exists a unique
  solution $u\in W^{2,1}(\Omega)\cap W^{1,1}_{\cD}(\Omega)$ of \eqref{prob:poisson}, and
  \begin{equation}
   \|D^2u\|_{L^1(\Omega)}
   +\|\delta_\Gamma^{-1}\nabla u\|_{L^1(\Omega)}
   \leq C\|f\|_{H^1_{\cD}(\Omega)},
   \label{eq:smooth-hardy-endpoint}
  \end{equation}
  where $C=C(\Omega)$.
  \item If $1<p<4/3$ and $f\in L^p(\Omega)$, there exists a unique
  solution $u\in W^{2,p}(\Omega)\cap W^{1,p}_{\cD}(\Omega)$ of \eqref{prob:poisson},
  and
  \begin{equation}
   \|D^2u\|_{L^p(\Omega)}
   +\|\delta_\Gamma^{-1}\nabla u\|_{L^p(\Omega)}
   \leq C\|f\|_{L^p(\Omega)},
   \label{eq:smooth-Lp-estimate}
  \end{equation}
  where $C=C(\Omega, p)$.
 \end{enumerate}
\end{theorem}
\begin{remark} \label{rmk-optimal}
    The endpoint $4/3$ is sharp. Indeed, in the upper half-plane the harmonic
function
\[
 u(r,\phi)=r^{1/2}\sin(\phi/2),
 \qquad 0<\phi<\pi,
\]
has zero Dirichlet data on one half of the boundary and zero Neumann data
on the other, while $|D^2u|\simeq r^{-3/2}$. Thus
$D^2u\notin L^{4/3}$ near the origin.
\end{remark}

From Theorem \ref{thm:poisson-smooth} and an extension theorem in \cite{MazyaMitreaShaposhnikova2010} that simultaneously extends the Dirichlet and Neumann data on the boundary, we obtain the $W^{2,p}$ well-posedness of \eqref{prob:inhomo}.
\begin{corollary}\label{cor-inhomo-1}
    Let $p \in (1,4/3)$, and suppose that the assumptions of Theorem \ref{thm:poisson-smooth} hold. Then, for every $f \in L^p(\Omega)$, $g_\cD \in W^{2-1/p,p}(\cD)$, and $g_\cN \in W^{1-1/p,p}(\cN)$, problem \eqref{prob:inhomo} admits a unique solution $u \in W^{2,p}(\Omega)$ satisfying
    \begin{equation*}
        \|u\|_{W^{2,p}(\Omega)}
        \leq C \left(
            \|f\|_{L^p(\Omega)}
            + \|g_\cD\|_{W^{2-1/p,p}(\cD)}
            + \|g_\cN\|_{W^{1-1/p,p}(\cN)}
        \right).
    \end{equation*} 
    The boundary conditions are understood in the trace sense.
\end{corollary}
Here, for $\mathcal E=\mathcal D$ or $\mathcal N$, we interpret $W^{s,p}(\mathcal E)$ as the restriction of $W^{s,p}(\partial\Omega)$ to $\mathcal E$, equipped with the usual quotient norm $\|g\|_{W^{s,p}(\mathcal E)} :=\displaystyle\inf\limits_{h|_{\mathcal E}=g} \|h\|_{W^{s,p}(\partial\Omega)}$.

\medskip

The regularity theory for scalar mixed problems on smooth domains goes
back to Shamir and Beir\~ao da Veiga. Shamir proved, among other results,
that weak solutions are $C^{(1/2)-}$ and obtained higher regularity in directions away from the critical one \cite{Shamir1968}. 
Beir\~ao da Veiga in \cite[Theorem~B]{MR372410} subsequently established the sharp
$W^{2,p}$ regularity, $p<4/3$, for $W^{1,(4/3)+}$ weak solutions on sufficiently regular domains with $f\in \cap_{r < \infty} L^r(\Omega)$, $g_\cD \in \cap_{r < \infty} W^{2-1/r,r}(\cD)$, and $g_\cN \in \cap_{r<\infty} W^{1-1/r,r}(\cN)$. 
This regularity result does not by itself establish solvability
for arbitrary $L^p$ source terms throughout $1<p<4/3$.
For instance, a first-order solvability argument seeking a solution
in $W^{1,q}_\cD(\Omega)$, $q>4/3$, requires control of the source in
its dual space $W_{\mathcal D}^{1,q'}(\Omega)$.
When $n\geq5$, the standard Sobolev embedding provides such control
for every $f\in L^p(\Omega)$, for some $q>4/3$, only when $p>\frac{4n}{3n+4}>1$.
Our results remove the additional integrability assumptions on the data and avoids the preliminary energy solution construction. This, in particular, allows us to obtain the $W^{2,p}$ solvability throughout the optimal $p$ range in all dimensions. It also allows a $C^{1,\alpha}$ boundary and a Reifenberg-flat interface, in contrast with the $C^3$ hypotheses in \cite[Theorem~B]{MR372410}.

There are also results in other function spaces. For instance, Savar\'e obtained the optimal $B^{3/2}_{2,\infty}$ regularity for energy solutions on $C^{1,1}$ domains \cite{Savare1997}. Mitrea and Mitrea established well-posedness in Sobolev--Besov spaces for mixed problems on creased Lipschitz domains \cite{MitreaMitrea2007}. The creasing hypothesis excludes smooth boundaries at the interface. More recently, optimal $W^{1,p}$ regularity and solvability have been obtained under Reifenberg-flat or Lipschitz assumptions by the second-named author in \cite{MR4232502,MR4261267}. 

\medskip

Our method also applies to nonsmooth domains. In this paper, we work on Lipschitz polyhedra, a prototypical class of domains with corners and edges. 

A bounded Lipschitz domain $\Omega\subset\bR^n$ is called a \emph{Lipschitz
polyhedron} if there are finitely many pairwise disjoint, relatively open
faces $F_1,\dots,F_J\subset\p\Omega$ such that each $F_j$ is a bounded
$(n-1)$-dimensional polytope contained in an affine hyperplane and $\p\Omega=\bigcup_{j=1}^J\overline{F_j}$.
No convexity is assumed. A boundary decomposition is \emph{facewise} if,
\begin{equation*}
    \text{for each}\,\, j,
    \quad
    \text{either}\,\,F_j \subset \cD \,\,\text{or}\,\, F_j\subset\cN.
\end{equation*}
For the polyhedral argument, we distinguish the mixed interface $\Gamma$ from the full lower-dimensional skeleton
\begin{equation}
 \Sigma:=\bigcup_{j=1}^J\p F_j,
 \qquad
 \delta_\Sigma(x):=\dist(x,\Sigma).
 \label{eq:polyhedral-skeleton}
\end{equation}
Thus $\Gamma\subset\Sigma$, but the inclusion may be strict when two adjacent faces carry the same boundary condition. Our second result gives the $W^{2,p}$ solvability on Lipschitz polyhedron in every dimension, for $p$ close to $1$.

\begin{theorem}
 \label{thm:poisson-polyhedron}
 Let $\Omega\subset\bR^n$ be a bounded and Lipschitz polyhedron with a
 facewise mixed decomposition $(\cD,\cN)$. Then there exists $p_0 = p_0(\Omega)>1$, such that
 the following assertions hold.
 \begin{enumerate}
  \item If $f\in H^1_{\cD}(\Omega)$, there exists a unique solution $u\in W^{2,1}(\Omega)\cap W^{1,1}_{\cD}(\Omega)$ of \eqref{prob:poisson}
  and
  \begin{equation}
   \|D^2u\|_{L^1(\Omega)}
   +\|\delta_\Sigma^{-1}\nabla u\|_{L^1(\Omega)}
   \leq C\|f\|_{H^1_{\cD}(\Omega)},
   \label{eq:polyhedral-hardy-endpoint}
  \end{equation}
   where $C=C(\Omega)$.
  \item If $1<p<p_0$ and $f\in L^p(\Omega)$, there exists a unique
  solution $u\in W^{2,p}(\Omega)\cap W^{1,p}_{\cD}(\Omega)$ of \eqref{prob:poisson}
  and
  \begin{equation}
   \|D^2u\|_{L^p(\Omega)}
   +\|\delta_\Sigma^{-1}\nabla u\|_{L^p(\Omega)}
   \leq C\|f\|_{L^p(\Omega)},
   \label{eq:polyhedral-Lp-estimate}
  \end{equation}
   where $C=C(\Omega,p)$.
 \end{enumerate}
\end{theorem}

Combined with an extension theorem, we obtain the $W^{2,p}$ well-posedness of \eqref{prob:inhomo} with facewise boundary data.
\begin{corollary}\label{cor-inhomo-2}
Suppose that the assumptions of
Theorem~\ref{thm:poisson-polyhedron} hold.
There exists $p_1=p_1(\Omega)>1$ such that the following
assertion holds for every $1<p<p_1$.
Let $f\in L^p(\Omega)$,
\[
 g_{\cD,j}:=g_\cD|_{F_j}\in W^{2-1/p,p}(F_j)
 \quad\text{for every }F_j\subset\cD,
\]
and
\[
 g_{\cN,j}:=g_\cN|_{F_j}\in W^{1-1/p,p}(F_j)
 \quad\text{for every }F_j\subset\cN.
\]
Assume the compatibility condition
\[
 \operatorname{Tr}_E g_{\cD,j}
 =
 \operatorname{Tr}_E g_{\cD,k}
\]
holds on every common $(n-2)$-dimensional face $E\subset\overline{F_j}\cap\overline{F_k}$ of two Dirichlet faces. 
Then problem~\eqref{prob:inhomo} admits a unique solution $u\in W^{2,p}(\Omega)$ satisfying
\[
 \|u\|_{W^{2,p}(\Omega)}
 \leq C\left(
   \|f\|_{L^p(\Omega)}
   +\sum_{F_j\subset\cD}
      \|g_{\cD,j}\|_{W^{2-1/p,p}(F_j)}
   +\sum_{F_j\subset\cN}
      \|g_{\cN,j}\|_{W^{1-1/p,p}(F_j)}
 \right),
\]
\end{corollary}
When $n=2$, the compatibility condition means agreement at common vertices of Dirichlet edges.

The choice of Lipschitz polyhedra is natural for two reasons. \textit{First}, additional geometric structure is needed for global $W^{2,p}$ estimates: such estimates cannot be expected on arbitrary Lipschitz domains, even under pure Dirichlet boundary conditions. See Dahlberg's classical counterexample \cite{MR544584}. \textit{Second}, polyhedral domains arise naturally in numerical analysis. To achieve optimal convergence rates on quasi-uniform meshes, traditional numerical approximation error estimates (for example: finite element approximations \cite{MR1930132}) frequently rely on $W^{2,2}$ regularity \cite{MR2373954}. At corners, edges, or mixed-boundary interfaces, such regularity may break down \cite[Examples~5.5.2 and~5.5.4]{MR2373954}. Our result provides a low-regularity framework.

\medskip

The analysis of elliptic problems on polyhedral domains has a long history. Following Kondrat\textquotesingle ev's foundational work, the singular behavior near corners, edges, and vertices has been studied extensively through operator pencils, singular expansions, and weighted Sobolev spaces; see
\cite{Kondratiev1967,MR3396210,MR1147281,MazyaRossmann2003, MazyaRossmann2010, BacutaMazzucatoNistorZikatanov2010} and the references therein. 
These methods give precise geometry-dependent regularity criteria. For example, Dauge's unweighted regularity theorem on three-dimensional curvilinear polyhedra is formulated for variational solutions, with the assumption $p\geq6/5$ so that $L^p(\Omega)$ is embedded into the dual of the energy space $W^{1,2}_\cD(\Omega)$; see \cite[(3.1) and Theorem~3.2]{MR1147281}.
More generally, Maz'ya and Rossmann in \cite[Theorems~4.4 and~7.2]{MazyaRossmann2003} established weighted $L^p$ solvability on polyhedral cones and a Fredholm theory  on bounded three-dimensional polyhedral domains under spectral conditions.
In arbitrary dimension, B\u{a}cu\c{t}\u{a}, Mazzucato, Nistor, and Zikatanov \cite{BacutaMazzucatoNistorZikatanov2010} developed regularity and well-posedness in weighted $L^2$-Sobolev spaces. Their well-posedness theorem assumes, in particular, that no two adjacent faces carry Neumann conditions, although their regularity theorem does not require this restriction.

Our emphasis is on a direct, unweighted solvability result for arbitrary $f\in L^p(\Omega)$ with $p>1$ sufficiently close to $1$, together with the adapted $H^1_{\cD}\to W^{2,1}$ endpoint.
Theorem~\ref{thm:poisson-polyhedron} applies in every dimension and permits arbitrary facewise Dirichlet--Neumann decompositions, including adjacent Neumann faces, and therefore is complementary to the more detailed geometry-dependent regularity supplied by the spectral theory.

These estimates also have a direct consequence for the numerical approximation of solutions to \eqref{prob:poisson}. For instance, let $V_{h,\cD}$ be the space of continuous piecewise affine functions vanishing on $\cD$, on a shape-regular simplicial mesh that respects the boundary decomposition \cite{MR1930132}, and let $h$ denotes the maximum element diameter. The Scott--Zhang approximation theorem \cite[Theorem~4.1]{ScottZhang1990}, together with our Theorem~\ref{thm:poisson-polyhedron}, gives, for $1<p<p_0$,
$$
 \inf_{v_h\in V_{h,\cD}}
 \left(
  \|u-v_h\|_{L^p(\Omega)}
  +h\|\nabla(u-v_h)\|_{L^p(\Omega)}
 \right)
 \leq Ch^2\|f\|_{L^p(\Omega)}.
$$
At $p=1$, the Hardy-space endpoint gives the same bound with $Ch^2\|f\|_{H^1_{\cD}(\Omega)}$ on the right-hand side. 
Our results therefore provide regularity and approximation estimates for the numerical analysis of scalar diffusion problems. Extensions to planar Lam\'e system will also be discussed in Section~\ref{sec:lame}.

\medskip

The proofs of both theorems follow a common strategy. We first construct an adapted Whitney-type decomposition that reduces second-order estimates to weighted gradient estimates, capturing the codimension-two singular structure. From this, we then establish the endpoint $H^1_{\cD}\to W^{2,1}$ estimate using H\"older estimates and duality. Finally, we obtain the $W^{2,p}$ estimates by a real-variable interpolation argument. The range of $p$ comes from local estimates for homogeneous solutions: the optimal $W^{1,s}$ estimates, $s<4$, in the $C^{1,\alpha}$ setting, and a $W^{1,2+\epsi}$ estimate on Lipschitz polyhedra. Our approach avoids preliminary $W^{1,q}_{\cD}$ solvability for the original rough datum and, in the polyhedral case, explicit analysis of corner and edge spectra.

The paper is organized as follows. In Section~\ref{sec-preliminaries}, we introduce the notation, definitions, and function spaces used throughout the paper. Section~\ref{sec-whitney} develops a Whitney-type reduction that is central to our approach, while Section~\ref{sec-hardy} establishes the required estimates in endpoint Hardy spaces. In Section~\ref{sec-uniqueness}, we prove uniqueness, and in Section~\ref{sec-interpolation}, we establish an interpolation lemma and use it to obtain solvability and global $W^{2,p}$ estimates for the Poisson equation with $L^p$ data. Finally, in Section~\ref{sec:lame}, we extend our approach to the Lam\'e system of linear elasticity \cite{MR936420} in two dimensions with mixed boundaries and discuss the key challenges in three and higher dimensions.

\section{Definition and function spaces}
\label{sec-preliminaries}

\subsection{$C^{1,\alpha}$ domains and Reifenberg flat interface}

We use coordinates $x=(x_1,x')$, where
$x'=(x_2,\ldots,x_n)\in\bR^{n-1}$.
Write $B'_{R_0}=\{x'\in\bR^{n-1}:|x'|<R_0\}$.

\begin{definition}
 A bounded domain $\Omega\subset\bR^n$ is a Lipschitz domain with character $(R_0,M)$ if, for every $x_0 \in\partial\Omega$, there are orthogonal coordinates centered at $x_0$ in which
 $$
  \Omega\cap Q_{R_0}
  =\bigl\{(x_1,x')\in Q_{R_0}:x_1>\psi_{\mathcal{P}}(x')\bigr\},
 $$
 where $Q_{R_0}=(-R_0,R_0)\times B'_{R_0}$ and $\psi\in C^{0,1}(B'_{R_0})$ satisfies
 $$
  \psi(0)=x_0,
  \qquad
  \|\nabla\psi_{\mathcal{P}}\|_{L^\infty(B'_{R_0})}\leq M.
 $$
 For $0<\alpha\leq1$, we say that $\Omega$ is a $C^{1,\alpha}$ domain with character $(R_0,M)$ if these charts can be chosen with $\psi\in C^{1,\alpha}(B'_{R_0})$ and
 $$
  \|\nabla\psi_{\mathcal{P}}\|_{L^\infty(B'_{R_0})}
  +R_0^\alpha[\nabla\psi_{\mathcal{P}}]_{C^{0,\alpha}(B'_{R_0})}
  \leq M.
 $$
 The constants $R_0$ and $M$ are uniform in $x_0$.
\end{definition}

\begin{assumption}\label{asp-cork}
    We say that $(\cD,\cN)$ satisfies the Dirichlet corkscrew condition if there exist $R_0>0$ and $c\in(0,1)$ such that, for every
 $x_0 \in\Gamma$ and $0<r<R_0$, there is a point
 $y\in\cD\cap B_r(x_0)$ for which
 $$
  B_{cr}(y) \cap \p \Omega \subset \cD\cap B_r(x_0).
 $$
\end{assumption}

The following interface condition was introduced in \cite{MR4261267} and \cite{MR4232502}. For $n=2$, this implies the interface $\Gamma$ consists of finitely many points on $\p\Omega$.
\begin{assumption}[$\theta$-flat]
  Let $0<\theta<1$. We say that $(\cD,\cN)$ has a $\theta$-Reifenberg-flat interface, relative to $\partial\Omega$, if there exists some $R_0 > 0$, such that for every $x_0 \in \Gamma$ and
 $0<r<R_0$, boundary graph coordinates as above can be chosen so that
 \begin{align}
  &\Gamma \cap B_r
   \subset\partial\Omega\cap B_r\cap\{|x_2|\leq\theta r\},
   \label{eq:RF1}\\
  &\partial\Omega\cap B_r\cap\{x_2>\theta r\}
   \subset\cD,
   \label{eq:RF2}\\
  &\partial\Omega\cap B_r\cap\{x_2<-\theta r\}
   \subset\cN.
   \label{eq:RF3}
 \end{align}
\end{assumption}

When $\Gamma \subset \p\Omega$ is given in boundary graph coordinates by $\{x \in \p\Omega: x_2 = \varphi(x_3,\ldots,x_n)\}$ for some $\varphi \in C^1$, it is $\theta$-flat with an arbitrarily $\theta$ and $R=R(\theta)$.

Clearly, if $\Gamma$ is $\theta$-flat with an absolute small constant, then $(\cD, \cN)$ satisfies a corkscrew condition. 
In the following lemma, we show flatness of $\Gamma$ gives a quantitative control of the volume of neighborhoods of $\Gamma$. Recall $\delta_\Gamma(x):=\dist(x,\Gamma)$ for $x\in\bR^n$. For a measurable set $E$ with $0<|E|<\infty$, we write
$$
 \fint_E f:=\frac{1}{|E|}\int_E f.
$$
\begin{lemma} \label{lem-codim2-vol}
Let $\Omega$ be a bounded Lipschitz domain, and let $0<\epsilon<1$. There exists $\theta_0=\theta_0(n,M,\epsilon)>0$ such that, if the interface is $\theta$-Reifenberg flat with $\theta \in (0, \theta_0)$, then
\begin{equation}
  \fint_{B_r(x_0)}\delta_\Gamma(x)^{-s}\,dx\leq Cr^{-s},
  \quad
  \forall x_0 \in \overline{\Omega},\,\,\forall 0<r<\operatorname{diam}(\Omega), \,\,\forall 0<s<2-\epsilon.
  \label{eqn-260819-1032}
\end{equation}
\end{lemma}

The proof follows from a covering argument and \cite[Lemma~2.5]{MR3034453}. For details, see \cite[proof of Lemma~5.1]{MR4232502}.

\subsection{Function spaces and atoms} \label{sec-160831-0602}
For $1\leq p<\infty$, as usual, define 
\begin{equation*}
    W_{\cD}^{1,p}(\Omega) :=\overline{C_\cD^\infty(\Omega)},
    \,\,\text{with}\,\,
    C_\cD^\infty(\Omega):= \{\varphi \in C^\infty_c(\bR^n), \varphi=0\,\,\text{in a neighborhood of}\,\,\cD\},
\end{equation*}
where the closure is taken with respect to the usual $W^{1,p}$ norm. 

A function $u\in W^{1,p}_{\cD}(\Omega)$ is called a weak solution of \eqref{prob:poisson}, if
\begin{equation}
  \int_\Omega\nabla u\cdot\nabla\varphi\, dx
  =\int_\Omega f\varphi\, dx
  \qquad
  \forall\varphi\in C_{\cD}^\infty(\Omega).
  \label{eq:poisson-weak}
\end{equation}

For a Euclidean ball $B=B_r(x_0)$ centered at a point
$x_0\in\overline\Omega$, write $B_\Omega :=B\cap\Omega$.
We adapt the usual definition for atoms and atom Hardy spaces to mixed problems.
\begin{definition}
  Let $1<q\leq\infty$.  A measurable function $a$ is an adapted
  $(1,q)$-atom associated with $B$ if
  \begin{equation}
    \operatorname{supp} a\subset B_\Omega,
    \qquad
    \|a\|_{L^q(\Omega)}\leq |B_\Omega|^{1/q-1},
    \label{eq:atom-size}
  \end{equation}
and the cancellation condition
  \begin{equation}
    \int_\Omega a(x)\, dx=0
    \qquad\text{whenever }2B\cap\cD=\varnothing.
    \label{eq:atom-cancellation}
  \end{equation}
\end{definition}

\begin{definition}[Adapted atomic Hardy space]
 The space $H_{\cD}^1(\Omega)$ consists of all $f\in L^1(\Omega)$
 admitting a representation
 \begin{equation}
  f=\sum_{j=1}^\infty\lambda_j a_j
  \quad\text{in }L^1(\Omega),
  \qquad
  \sum_{j=1}^\infty|\lambda_j|<\infty,
  \label{eq:atomic-representation}
 \end{equation}
 where the $a_j$ are adapted $(1,\infty)_{\cD}$-atoms.  Its norm is
 \begin{equation}
  \|f\|_{H_{\cD}^1(\Omega)}
  :=\inf\sum_j|\lambda_j|,
  \label{eq:atomic-norm}
 \end{equation}
 where the infimum is taken over all such representations.
\end{definition}

We shall also use adapted $(1,s)_{\cD}$-atoms in the interpolation
argument.  The following standard reduction allows us to do so without
changing the Hardy space.

\begin{lemma}
\label{lem:atom-decomposition}
Let $\Omega\subset\mathbb R^n$ be a bounded Lipschitz domain with
$\emptyset\neq \mathcal D\subset\partial\Omega$. Suppose $a$ is an $(1,s)_{\mathcal D}$-atom for some $s\in (1,\infty)$. Then, there exist $(1,\infty)_{\mathcal D}$-atoms $\{a_j\}_j$ and $\lambda_j \in \bR$, such that
\[
    a=\sum_{j=1}^\infty \lambda_j a_j
    \qquad\text{in }L^1(\Omega)
    \quad
    \text{and}
    \quad
    \sum_{j=1}^\infty |\lambda_j|\leq C.
\]
Here $C$ depends only on $n$, $s$, and the Lipschitz character of $\Omega$. 
\end{lemma}

The proof is a simple adaption of the standard Coifman--Weiss decomposition of mean-zero atoms. We include here for completeness.

\begin{proof}
Since $\Omega$ is Lipschitz, $(\Omega,|\cdot|,dx)$ is a space of
homogeneous type.  If $2B\cap\cD=\varnothing$, then $a$ is an ordinary
mean-zero $(1,s)$-atom, and the conclusion follows from the
Coifman--Weiss decomposition; see \cite[Theorem~A]{MR447954}.

Suppose next that $2B\cap\cD\neq\varnothing$.  Put
\[
 m_B:=\fint_{B_\Omega}a\,dx,
 \qquad
 a_0:=m_B\chi_{B_\Omega},
 \qquad
 \widetilde a:=\frac12(a-a_0).
\]
H\"older's inequality gives $|m_B|\leq|B_\Omega|^{-1}$, so $a_0$ is
an adapted $(1,\infty)_{\cD}$-atom.  Also,
\[
 \int_\Omega\widetilde a\,dx=0,
 \qquad
 \|\widetilde a\|_{L^s(\Omega)}
 \leq |B_\Omega|^{1/s-1}.
\]
Thus $\widetilde a$ is an ordinary mean-zero $(1,s)$-atom.  Applying
the Coifman--Weiss decomposition to $\widetilde a$ and using
$a=a_0+2\widetilde a$ proves the lemma.
\end{proof}

As a useful consequence, if $1<s<\infty$ and $f$ is an adapted $(1,s)_\cD$-atom, then $f \in H^1_\cD$ with
$\|f\|_{H^1_\cD (\Omega)} \leq C_s$.

\section{A Whitney reduction}
\label{sec-whitney}

The following proposition reduces the local $W^{2,p}$ estimate to
a weighted estimate for the gradient. We treat the two classes of
domains appearing in our main theorems.

\begin{proposition}
 \label{prop:weighted-Whitney}
Let $1<p<\infty$, and suppose that one of the following holds:
\begin{enumerate}
    \item $\Omega\subset\bR^n$, $n\geq2$, is a bounded $C^{1,\alpha}$ domain with $\alpha>1-1/p$, and $\Sigma=\Gamma$ is the Dirichlet--Neumann interface;
    \item $\Omega\subset\bR^n$, $n\geq2$, is a bounded Lipschitz polyhedron, and $\Sigma$ is its full skeleton. Each open $(n-1)$-dimensional face carries either the Dirichlet or the Neumann boundary condition.
\end{enumerate}
There exist constants $r_0>0$, $\kappa>1$, and $C>0$ such that, if $f\in L^p(\Omega)$ and $u\in W^{1,1}_\cD(\Omega) \cap W^{2,p}_{loc}(\overline\Omega \setminus \Sigma)$ is a weak solution of \eqref{prob:poisson}, then
\begin{equation}
\label{eq:weighted-Whitney}
    \fint_{Q\cap\Omega}|D^2u|^p
    \leq C\fint_{\kappa Q\cap\Omega}
    \left(
        |f|^p+
        \bigl(\delta_\Sigma^{-p}+\ell(Q)^{-p}\bigr)
        |\nabla u|^p
    \right)
\end{equation}
for every cube $Q$ centered in $\overline{\Omega}$ with $\ell(Q)<r_0$.  The constant $\kappa$ depends only on $n$ and the Lipschitz character of $\Omega$; $C$ and $r_0$ may additionally depend on $p$ and the $C^{1,\alpha}$ character in case (a).
\end{proposition}
Here, $\ell(Q)$ represents the side length of $Q$ and $\delta_\Sigma(x) = \operatorname{dist}(x,\Sigma)$.
In the proposition, only local regularity away from $\Sigma$ is required. The estimate follows by summing over Whitney cubes, and, in particular, $\infty \leq \infty$ is allowed.

\medskip

The argument uses two properties shared by these domains: the boundary is locally a Lipschitz graph, and local $W^{2,p}$ estimates hold away from $\Sigma$. In case~{\rm (1)}, no additional assumption on $\Gamma$ is needed.

More precisely, for small cubes $Q$ in a coordinate neighborhood with $2Q\cap\Sigma = \emptyset$,
\begin{equation}
\label{eq:local-W2p-away-Sigma}
    \fint_{Q\cap\Omega}|D^2u|^p
    \leq C\fint_{2Q\cap\Omega}
    \left(|f|^p+\ell(Q)^{-p}|\nabla u|^p\right).
\end{equation}

The interior case $2Q \subset \Omega$ is standard. For Lipschitz polyhedra, the boundary estimates reduce locally to the standard half-space estimates. For $C^{1,\alpha}$ domains with $\alpha>1-1/p$, we refer to \cite{MR569375,MR4631039} for the Dirichlet problem and \cite{MR4387172} for the Neumann problem. The latter follows from the more general oblique derivative estimates there, since the unit normal to a $C^{1,\alpha}$ boundary belongs to $C^\alpha$. For this reduction, boundary regularity beyond graphical domain is used only to obtain these local estimates.

The proof is mainly by constructing a Whitney-type decomposition. Using the local graph representation, we cover the region near the boundary by ``lift-up cylinders'' obtained from a surface Whitney decomposition away from $\Sigma$, and cover the remaining region by interior, ``non-tangential'' cubes relative to $\Sigma$. The desired estimate follows from applying local $W^{2,p}$ estimates on each cube, noting they are now away from $\Sigma$.

\begin{proof}
We first consider a cube $Q$ centered at $x_0\in\Sigma$. By choosing
$r_0$ sufficiently small, after an orthogonal change of coordinates centered
at $x_0$, we have
\[
 \Omega\cap \kappa_0 Q
 =\{(x',x_n)\in\kappa_0 Q:x_n>\psi(x')\},
\]
where $\kappa_0>1$ is a sufficiently large fixed constant and $\psi$ is
Lipschitz. Let $Q'$ be the projection of $Q$ onto the coordinate plane and set
\[
 \Sigma':=\{x'\in \kappa_0 Q':(x',\psi(x'))\in\Sigma\}.
\]

We first estimate near the boundary. Fix $1<\kappa_1<\kappa_0$ sufficiently large and take a standard Whitney decomposition of $\kappa_1Q'\setminus\Sigma'$. Collect those cubes $I_j$ having nonempty intersection with $Q'$. These cubes have pairwise disjoint interiors and $Q'\setminus\Sigma'\subset\bigcup_j I_j$. Since the center of $Q'$ belongs to $\Sigma'$ and $\kappa_1$ is sufficiently large, they may be chosen so that,
writing
$\ell_j:=\ell(I_j)$,
\[
 \frac{1}{16\sqrt{n-1}}\dist(I_j,\Sigma')
 \leq \ell_j
 \leq \frac{1}{4\sqrt{n-1}}\dist(I_j,\Sigma').
\]
In particular, $4I_j\cap\Sigma'=\emptyset$. For details of this decomposition, see, for example, \cite[Appendix~J]{MR2445437}. Define the lift-up map
\[
 \Phi(y',t):=(y',\psi(y')+t).
\]
The map $\Phi$ is bi-Lipschitz, with
$\dist(\Phi(y',t),\partial\Omega)\simeq t$. Put
\[
 T_j:=\Phi(I_j\times(0,\ell_j)),
 \qquad
 T_j^*:=\Phi(2I_j\times(0,2\ell_j)).
\]
Since $4I_j\cap\Sigma'=\emptyset$, we have
$\delta_\Sigma(x)\simeq\ell_j$ for all $x\in T_j^*$. By \eqref{eq:local-W2p-away-Sigma},
\begin{equation}\label{est-bdycube}
 \int_{T_j}|D^2u|^p
 \leq C\int_{T_j^*}
 \left(|f|^p+\delta_\Sigma^{-p}|Du|^p\right).
\end{equation}

\begin{figure}
\begin{tikzpicture}[scale=1, line cap=round, line join=round]

\tikzset{
  boundary/.style={black, line width=1.25pt},
  coneedge/.style={red!85!black, line width=1.35pt},
  conefill/.style={red!12, draw=none},
  liftfill/.style={blue!10, draw=none},
  liftedge/.style={blue!80!black, line width=1.2pt},
  whitney/.style={draw=purple!80!black, line width=1.2pt},
  label/.style={font=\small},
}

\coordinate (P) at (0,0);
\coordinate (BL) at (-7,-2.1);
\coordinate (BR) at ( 7,-2.8);

\draw[boundary] (BL) -- (P) -- (BR);
\node[label,below left=2pt] at (BL) {$\cD$};
\node[label,below right=2pt] at (BR) {$\cN$};
\node[label] at (6.0,0) {$\Omega$};

\fill (P) circle (1.3pt);
\node[label,below=3pt] at (P) {$\Sigma$};

\coordinate (CL) at (-7.2, 1.8);
\coordinate (CR) at ( 6.0, 1.2);

\fill[conefill] (P) -- (CL) -- (CR) -- cycle;
\draw[coneedge] (P) -- (CL);
\draw[coneedge] (P) -- (CR);
\node[label,red!85!black] at (3.2,1.4) {$\mathcal{G}_\alpha$};


\coordinate (D1) at ($(P)!0.60!(BL)$);
\coordinate (D2) at ($(P)!0.40!(BL)$);
\draw[liftedge] (D1) -- (D2);
\node[label,blue!80!black,below left=2pt] at ($(D1)!1!(D2)$) {$I_j$};

\coordinate (D1up) at ($ (D1)!1!90:(D2) $);
\coordinate (D2up) at ($ (D2)!1!-90:(D1) $); 

\fill[liftfill] (D1) -- (D2) -- (D2up) -- (D1up) -- cycle;
\draw[liftedge] (D1) -- (D2) -- (D2up) -- (D1up) -- cycle;
\node[label,blue!80!black] at ($(D1up)!0.5!(D2)$) {$T_j$};

\coordinate (N1) at ($(P)!0.3!(BR)$);
\coordinate (N2) at ($(P)!0.2!(BR)$);
\draw[liftedge] (N1) -- (N2);

\coordinate (N1up) at ($ (N1)!1!-90:(N2) $);
\coordinate (N2up) at ($ (N2)!1!90:(N1) $);

\fill[liftfill] (N1) -- (N2) -- (N2up) -- (N1up) -- cycle;
\draw[liftedge] (N1) -- (N2) -- (N2up) -- (N1up) -- cycle;


\def\q{1}

\coordinate (Q1c) at ($(D1up)!0.55!(D2up) + (0, 0.35)$);  
\draw[whitney] ($(Q1c)+(-0.8*\q,-0.8*\q)$) rectangle ($(Q1c)+(0.8*\q,0.8*\q)$);
\node[label,purple!80!black,above=2pt] at (Q1c) {$\widetilde Q_j$};

\coordinate (Q2c) at ($(N1up)!0.2!(N2up)+(0.2,0.2)$);
\draw[whitney] ($(Q2c)+(-0.5*\q,-0.5*\q)$) rectangle ($(Q2c)+(0.5*\q,0.5*\q)$);

\coordinate (Q2c) at ($(-0.4,0.6)$);
\draw[whitney] ($(Q2c)+(-0.3*\q,-0.3*\q)$) rectangle ($(Q2c)+(0.3*\q,0.3*\q)$);

\end{tikzpicture}
\caption{Decomposition of $\Omega$}
\label{Figure1}
\end{figure}
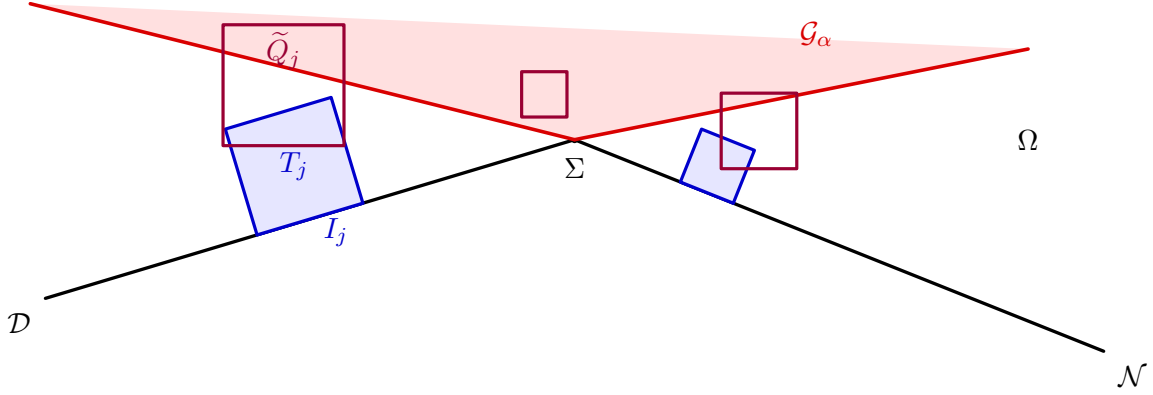

Next, we estimate on the ``non-tangential region''. Define the non-tangential region
relative to $\Sigma$ by
\[
 \mathcal G_\alpha
 :=\{x\in\Omega:\delta_\Sigma(x)
 \leq(1+\alpha)\dist(x,\partial\Omega)\},
\]
where $\alpha>0$ will be chosen later. Take a standard interior Whitney
decomposition $\{\widetilde Q_k\}$ of $\Omega$ and collect those cubes with
\[
 \widetilde Q_k\cap\mathcal G_\alpha\cap Q\neq\emptyset.
\]
For each such cube,
\begin{equation}\label{est-intcube}
 \delta_\Sigma(x)
 \simeq_\alpha \dist(x,\partial\Omega)
 \simeq \ell(\widetilde Q_k),
 \qquad x\in2\widetilde Q_k,
\end{equation}
where $ A \simeq_\alpha B$ means $C_1 B \leq A \leq C_2 B$ and $C_1$ and $C_2$ are constants depending on $\alpha$. Indeed, the lower bound follows from
$\delta_\Sigma\geq\dist(\,\cdot\,,\partial\Omega)$. For the upper bound,
choose $y\in\widetilde Q_k\cap\mathcal G_\alpha$. Then, for
$x\in2\widetilde Q_k$,
\[
 \delta_\Sigma(x)
 \leq |x-y|+\delta_\Sigma(y)
 \leq C\ell(\widetilde Q_k)
 +(1+\alpha)\dist(y,\partial\Omega)
 \leq C_\alpha\ell(\widetilde Q_k).
\]
The interior estimate therefore yields
\[
 \int_{\widetilde Q_k}|D^2u|^p
 \leq C\int_{2\widetilde Q_k}
 \left(|f|^p+\delta_\Sigma^{-p}|Du|^p\right).
\]

We claim that $\Omega\cap Q$ is covered by the union of boundary lift-up cubes $T_j$ and the selected interior Whitney cubes $\widetilde Q_k$ (see Figure \ref{Figure1}). To see this,
take $x=\Phi(y',t)\in Q\cap\Omega$ which does not belong to any $T_j$, and let
\[
 \mathcal{P}(x):=\Phi(y',0).
\]
If $\mathcal{P}(x)\in\Sigma$, then
\[
 \delta_\Sigma(x)\leq |x-\mathcal{P}(x)|
 \leq C\dist(x,\partial\Omega).
\]
If $\mathcal{P}(x)\notin\Sigma$, let $I_j$ contain $y'$. Since $x\notin T_j$, we have
$t>\ell_j$. The Whitney property gives
\[
 \delta_\Sigma(\mathcal{P}(x))\leq C\ell_j<Ct,
\]
and hence again
\[
 \delta_\Sigma(x)
 \leq |x-\mathcal{P}(x)|+\delta_\Sigma(\mathcal{P}(x))
 \leq C\dist(x,\partial\Omega).
\]
Thus $x\in\mathcal G_\alpha$ when $\alpha$ is sufficiently large, and the
interior Whitney cube containing $x$ is one of the selected cubes.

The families $\{T_j^*\}$ and $\{2\widetilde Q_k\}$ have uniformly bounded
overlap, and all these sets are contained in a fixed dilation of $Q$. Adding
the preceding estimates gives
\begin{align} \label{eq:estimate}
     \int_{Q}|D^2u|^p
 \leq C\int_{\kappa_0Q}
 \left(|f|^p+\delta_\Sigma^{-p}|Du|^p\right).
\end{align}

\medskip

It remains to consider a general cube $Q$, not necessarily centered at $\Sigma$. If
$2Q\cap\Sigma=\emptyset$, the desired estimate follows directly from
\eqref{eq:local-W2p-away-Sigma}. Otherwise, choose $y_0\in2Q\cap\Sigma$ and a cube $Q^*$
centered at $y_0$, with side length comparable to $\ell(Q)$, such that
$Q\subset Q^*$. Applying the preceding estimate \eqref{eq:estimate} to $Q^*$, and then choosing
$\kappa$ sufficiently large so that $\kappa_0Q^*\subset\kappa Q$, proves
\eqref{eq:weighted-Whitney}.
\end{proof}

\section{Estimate in endpoint Hardy spaces} \label{sec-hardy}

\subsection{A $W^{1,2+\epsilon}$ estimate}
Consider \eqref{prob:poisson} with $f \in L^p(\Omega)$ with $p > 2n/(n+2)$. By Sobolev embedding, $f \in W^{-1,2}(\Omega)$, it follows that there is a unique solution $u \in W^{1,2}_\cD(\Omega)$. In this section, under mild conditions on $\Omega$ and $\Gamma$, we prove a $W^{1,2+\epsilon}$ estimate.
\begin{lemma}\label{lem-260819-1105}
Let $\Omega \subset \bR^n$ be a bounded Lipschitz domain. Suppose that $(\cD, \cN)$ satisfies the corkscrew condition in Assumption \ref{asp-cork}. Then there exists some $\epsi > 0$, such that for every $p \in (2_*, (2+\epsi)_*)$, the unique solution $u \in W^{1,2}_\cD(\Omega)$ of \eqref{prob:poisson} satisfies $u \in W^{1,p^*}_\cD(\Omega)$ and
\begin{equation*}
    \|\nabla u\|_{L^{p^*}(\Omega)} \leq C\|f\|_{L^p(\Omega)}.
\end{equation*}
\end{lemma}
Here, $p_* = np/(n+p)$ and $p^* = np/(n-p)$.

\medskip

This follows from a standard $W^{1,2+\epsilon}$ estimate for equation with divergence-form right-hand side, duality, and the Sobolev-Poincar\'e inequality. We give a sketch below.

\begin{proof}
    For any $F\in C^\infty_c(\Omega)$, let $v_F \in W^{1,2}_\cD(\Omega)$ solve the dual equation
\begin{equation*}
    -\Delta v = \operatorname{div}(F)\,\,\text{in}\,\,\Omega,
    \quad
    v=0\,\,\text{on}\,\,\cD,
    \quad
    \p v/\p \nu = 0\,\,\text{on}\,\,\cN.
\end{equation*}
A standard argument implies that there exists some $p_1<2$, such that
\begin{equation*}
    \| \nabla v_F\|_{L^q(\Omega)} \leq C \|F\|_{L^q(\Omega)},
    \quad
    \forall q \in (p_1,p_1')
\end{equation*}
where $p_1'=p_1/(p_1-1) > 2$. Here, the corkscrew condition is used in order to apply the Sobolev-Poincar\'e inequality. See for example, \cite[Lemma~3.4]{MR4261267}. Testing \eqref{prob:poisson} by $v_F$, we have
\begin{equation} \label{eqn-260922-0152}
    \int_\Omega \nabla u \cdot F
    =
    \int_\Omega f v_F.
\end{equation}
Since $F\in C^\infty_c(\Omega)$ is arbitrary, using \eqref{eqn-260922-0152} and the Sobolev-Poincar\'e inequality, we obtain the desired estimate.
\end{proof}

\subsection{Estimate with atom data}

In this section, we consider \eqref{prob:poisson} on $C^{1,\alpha}$ domains with $f=a$, an $(1,\infty)_\cD$-atom. Since $a \in L^\infty(\Omega) \subset W^{-1,2}_{\cD}(\Omega)$, there exists a unique weak solution $u \in W^{1.2}_\cD(\Omega)$.

\begin{proposition}
 \label{prop:atom-data}
 Let $\Omega\subset\bR^n$ be a bounded domain satisfying one of the following two conditions:
 \begin{enumerate}
    \item $\Omega \in C^{1,\alpha}$ for some $\alpha > 0$ with the Dirichlet--Neumann interface $\Gamma$ is $\theta_0$-flat for some small $\theta_0$;
    \item $\Omega$ is a Lipschitz polyhedron with a facewise Dirichlet--Neumann decomposition and $\Sigma$ being the full polyhedron skeleton.
\end{enumerate}
Let $u\in W_{\cD}^{1,2}(\Omega)$ solves \eqref{prob:poisson} with $f=a$, an adapted $(1,\infty)_{\cD}$-atom. Then
 \begin{equation}
  \|D^2u\|_{L^1(\Omega)}
  +\|\delta^{-1}\nabla u\|_{L^1(\Omega)}
  \leq C.
  \label{eq:atom-data-estimate}
 \end{equation}
 Here $\delta = \delta_\Gamma$ under condition (a) and $\delta = \delta_\Sigma$ under condition (b)
\end{proposition}

\begin{proof}
Let $B=B_R(x_0)$ be the ball associated with $a$. By a simple covering, we may assume that $R$ is smaller than the fixed coordinate radius.

Since $|B\cap\Omega|\simeq R^n$ and $\|a\|_{L^\infty(\Omega)} \leq |B\cap\Omega|^{-1} \leq CR^{-n}$, by Lemma \ref{lem-260819-1105}, there exists $\epsi>0$, such that for every $p \in (2,2+\epsi)$,
\begin{equation}
 \|\nabla u\|_{L^p(\Omega)}
 \leq C\|a\|_{L^{p_*}(\Omega)}
 \leq C\|a\|_{L^\infty(\Omega)} |B\cap\Omega|^{1/p_*}
 \leq C|B\cap\Omega|^{-1 + 1/p_*}
 \leq C R^{1+n/p-n}.
 \label{eq:atom-gradient-meyers}
\end{equation}

Fix one such $p$. Choose $\epsi_0\in(0,2-p')$ and then
$\beta\in(p',2-\epsi_0)$ sufficiently close to $p'$ so that, with
\begin{equation}
 \frac1q:=\frac1p+\frac1\beta<1,
 \qquad 1-\frac1q<\alpha.
 \label{eq:atom-exponent-choice}
\end{equation}
Under condition (a), we decrease $\theta_0$, depending on $\epsi_0$, such that Lemma \ref{lem-codim2-vol} holds. Under condition (b), the polyhedral structure of $\Sigma$ gives $\int_{B_r(x_0)}\delta_\Sigma(x)^{-\beta}\,dx
 \leq C_\beta r^{n-\beta}$, for all $0<\beta<2$.
Thus the same argument applies with $\delta=\delta_\Sigma$, without a flatness assumption.

\medskip

\noindent\textbf{Step 1: Estimate near the support.}
By H\"older's inequality, we have
\begin{equation} \label{eqn-260820-0831}
     \|\delta^{-1}\nabla u\|_{L^q(B_{8R}(x_0)\cap\Omega)}
    \leq 
    \|\delta^{-1}\|_{L^\beta(B_{8R}(x_0)\cap\Omega)} \|\nabla u\|_{L^p(B_{8R}(x_0)\cap\Omega)}.
\end{equation}
From $\beta < 2-\epsi_0$, Lemma \ref{lem-codim2-vol}, and \eqref{eq:atom-gradient-meyers}, we further have
\begin{align*}
    \text{RHS of \eqref{eqn-260820-0831}}
    \leq
    C R^{n/\beta-1}R^{1+n/p-n}
 =C R^{n/q-n}.
\end{align*}
Since $p>q$, $\|\nabla u\|_{L^q(B_{8R}(x_0)\cap\Omega)}
 \leq C R^{n/q-n/p}\|\nabla u\|_{L^p(\Omega)}$.
It follows that
\begin{equation}
 \|\delta^{-1}\nabla u\|_{L^q(B_{8R}(x_0)\cap\Omega)}
 +R^{-1}\|\nabla u\|_{L^q(B_{8R}(x_0)\cap\Omega)}
 \leq C R^{n/q-n}.
 \label{eq:atom-near-Lq}
\end{equation}
The weighted Whitney estimate in Proposition \ref{prop:weighted-Whitney} therefore gives
\[
 \|D^2u\|_{L^q(B_{4R}(x_0)\cap\Omega)}
 \leq 
 \|(\delta^{-1}+ R^{-1})\nabla u\|_{L^q(B_{8R}(x_0)\cap\Omega)} +\|a\|_{L^q(\Omega)}
 \leq
 C R^{n/q-n}.
\]
Applying H\"older's inequality again, we have
\begin{equation}
 \int_{B_{4R}(x_0)\cap\Omega}
 \bigl(|D^2u|+\delta^{-1}|\nabla u|\bigr)\,dx
 \leq C.
 \label{eq:atom-near-L1}
\end{equation}

\noindent\textbf{Step 2. $L^2$ decay for gradient.}
Set $r_k=2^kR$ and $A_k=\bigl(B_{r_{k+1}}(x_0)\setminus B_{r_k}(x_0)\bigr)$. We prove
\begin{equation}
 \|\nabla u\|_{L^2(A_k)}
 \leq C2^{-k\gamma}r_k^{1-n/2}.
 \label{eq:atom-L2-decay}
\end{equation}

For this, take $F\in L^2(A_k;\bR^n)$ with $\|F\|_{L^2(A_k)}=1$ and let
$\phi_F\in W_{\cD}^{1,2}(\Omega)$ solve
\[
 -\Delta \phi_F = -\operatorname{div}(F)\,\,\text{in}\,\,\Omega,
 \quad
 \phi_F = 0\,\,\text{on}\,\,\cD,
 \quad
 \p \phi_F/\p \vec{\nu} = F\cdot\nu\,\,\text{on}\,\,\cN.
\]
We have $\|\nabla\phi_F\|_{L^2(\Omega)}\leq C$.

Since $\phi$ satisfies the Laplace equation with homogeneous mixed boundary condition in $B \cap \Omega$, the local H\"older estimate (cf., \cite[Theorem~3.1]{MR3034453}) together with the Poincar\'e inequality gives, for some $\gamma>0$,
\begin{align*}
 \fint_{B\cap\Omega}|\phi_F-c_B|\,dx
 &\leq C R^\gamma [\phi_F]_{C^\gamma(2 B\cap\Omega)}
 \\&\leq C\left(\frac R{r_{k-1}}\right)^\gamma r_{k-1}
 \left(\fint_{B_{r_{k-1}}(x_0)\cap\Omega}|\nabla\phi_F|^2\,dx\right)^{1/2}\\
 &\leq C2^{-k\gamma}r_k^{1-n/2} \|\nabla \phi_F\|_{L^2(\Omega)} 
 \leq C2^{-k\gamma}r_k^{1-n/2},
\end{align*}
where $c_B=0$ if $2B\cap\cD\neq\emptyset$, and $c_B = \fint_{B\cap \Omega} \phi_F$ otherwise.
Since $\|a\|_{L^{\infty}(\Omega)} \leq|B\cap\Omega|^{-1}$ and the average of $a$ equals to zero if $2B\cap \cD = \emptyset$, we have
\[
 \left|\int_\Omega a\phi_F\,dx\right|
 =\left|\int_\Omega a(\phi_F-c_B)\,dx\right|
 \leq C2^{-k\gamma}r_k^{1-n/2}.
\]
By duality, we have
\[
 \int_{A_k}F\cdot\nabla u\,dx
 =\int_\Omega F\cdot \nabla u\,dx
 =\int_\Omega\nabla\phi_F\cdot\nabla u\,dx
 =\int_\Omega a\phi_F\,dx.
\]
Taking the supremum over $F$ proves \eqref{eq:atom-L2-decay}.

\medskip
\noindent \textbf{Step 3. Decay of $D^2u$ and $\delta^{-1}\nabla u$.} For $k \geq 2$, denote $A_k^* = \cup_{j=-1}^1 A_{k+j}$.

Recall $p \in (2,2+\epsi)$. By a standard local reverse H\"older estimate and \eqref{eq:atom-L2-decay}, we obtain
\begin{equation}
 (\fint_{A_k} |\nabla u|^p)^{1/p}
 \leq C (\fint_{A_k^*}  |\nabla u|^2)^{1/2}
 \leq C2^{-k\gamma}r_k^{1-n}.
 \label{eq:atom-Lp-decay}
\end{equation}

From this, an application of H\"older's inequality and Lemma \ref{lem-codim2-vol} as in Step 1 gives
\begin{equation}
 \|\delta^{-1}\nabla u\|_{L^q(A_k)}
 +r_k^{-1}\|\nabla u\|_{L^q(A_k)}
 \leq C2^{-k\gamma}r_k^{n/q-n}.
 \label{eq:atom-annular-Lq}
\end{equation}
Since $a=0$ on $A_k^*$, the weighted Whitney estimate in Proposition \ref{prop:weighted-Whitney} gives
\[
 \|D^2u\|_{L^q(A_k)}
 \leq C2^{-k\gamma}r_k^{n/q-n}.
\]

Therefore, H\"older's inequality implies, for any $k \geq 2$,
\begin{equation}
 \int_{A_k}
 \bigl(|D^2u|+\delta^{-1}|\nabla u|\bigr)\,dx
 \leq C2^{-k\gamma}.
 \label{eq:atom-annular-L1}
\end{equation}

Adding this up for $k \geq 2$ and \eqref{eq:atom-near-L1}, we reach the desired estimate.
\end{proof}

\subsection{Solvability with Hardy data}

From Proposition \ref{prop:atom-data}, the solvability when $f \in H^1_\cD$ easily follows. Indeed, let $f = \sum_j \lambda_j a_j$ be a atom decomposition with $\sum_j |\lambda_j| \leq 2 \|f\|_{H^1_\cD}$. Since $f_n := \sum_{j\leq n}\lambda_j a_j \in L^\infty(\Omega)$, we can solve \eqref{prob:poisson} for a solution $u_n \in W^{1,2}_\cD(\Omega)$. By Proposition \ref{prop:atom-data}, we have $u_n \in W^{2,1}(\Omega)$ forms a Cauchy sequence under the usual $W^{2,1}$ norm. Let $u \in W^{2,1}$ be the limit. Clearly, $u$ is a weak solution to \eqref{prob:poisson} and satisfies the desired estimate.

\section{Uniqueness}
\label{sec-uniqueness}

In this section, we establish uniqueness in the endpoint space $W^{2,1}(\Omega)$ that works both on $C^{1,\alpha}$ domains with mild assumptions on $\Gamma$ and Lipschitz polyhedrons with facewise decomposition. Since a solution in this space need not belong to the energy space $W^{1,2}_{\cD}(\Omega)$, the standard energy argument does not apply directly.

\begin{proposition}
 \label{prop:W21-uniqueness}
 Let $\Omega\subset\bR^n$ be a bounded, connected Lipschitz domain. Suppose that the decomposition $(\cD, \cN)$ satisfies a corkscrew condition as in Assumption \ref{asp-cork} and $\Gamma$ satisfies an integral bound in \eqref{eqn-260819-1032} with a sufficiently small $\epsi >0$. If
 \[
  u\in W^{2,1}(\Omega)\cap W^{1,1}_{\cD}(\Omega)
 \]
 solves \eqref{prob:poisson} with $f=0$, then $u\equiv0$.
\end{proposition}

For the proof, we reduce it to the uniqueness for the $L^1$-mixed problem established in \cite[Theorem~1.1(a)]{MR3034453}.  More precisely, for $x\in\p\Omega$, let 
\[
 \mathbb{N}(h)(x):=\sup_{y\in\gamma(x)}|h(y)|,
\]
where $\gamma(x):=\{y \in \Omega: |y-x| < (1+\alpha_0) \operatorname{dist}(y,\p\Omega)\}$ is the non-tangential cone with a fixed apperture $\alpha_0 > 0$. By \cite[Theorem~1.1(a)]{MR3034453}, to show $u\equiv 0$, it suffices to show $\mathbb{N}(u) , \mathbb{N} (\nabla u) \in L^1(\p\Omega)$. This follows from the following lemma.
\begin{lemma}
 \label{lem:W11-nontangential}
 Let $\Omega$ be a bounded Lipschitz domain. If
 $h\in W^{1,1}(\Omega)$ is harmonic in $\Omega$, then
 \begin{equation}
  \|\mathbb{N}(h)\|_{L^1(\p\Omega)}
  \leq C\|h\|_{W^{1,1}(\Omega)}.
  \label{eq:W11-nontangential}
 \end{equation}
\end{lemma}

By Lemma \ref{lem:W11-nontangential} applied to $u$ and $\nabla u$, Proposition \ref{prop:W21-uniqueness} follows from \cite[Theorem~1.1(a)]{MR3034453}. Indeed, $u \in W^{1,1}_\cD$ gives $u = 0$ on $\cD$. By $\nabla u\in W^{1,1}(\Omega)$, the weak formulation implies $\nu\cdot\nabla u=0$. Both of these are in the sense of trace. Now, since the non-tangential limits agree almost everywhere with the Sobolev traces, we obtain that $u$ satisfies the homogeneous boundary conditions in the sense of the non-tangential limit required by \cite[Theorem~1.1(a)]{MR3034453}.

We are left to show Lemma \ref{lem:W11-nontangential}. This is an easy consequence of Dahlberg's Lusin area integral theorem. We include the proof for completeness.
\begin{proof}
 For $x\in\p\Omega$, let
 \[
  S(h)(x):=
  \left(
   \int_{\gamma(x)}|\nabla h(y)|^2
   \dist(y,\p\Omega)^{2-n}\,dy
  \right)^{1/2}.
 \]
 Fix $X_* \in \Omega$.
 By Dahlberg's Lusin area integral theorem
 \cite[Theorem~1]{Dahlberg1980},
 \begin{equation}
  \|\mathbb{N}(h-h(X_*))\|_{L^1(\p\Omega)}
  \leq C\|S(h)\|_{L^1(\p\Omega)}.
  \label{eq:Dahlberg-area}
 \end{equation}
Take a Whitney decomposition $\Omega = \cup_i Q_i$. For each $Q_i$, let $\Delta_{Q_i}:= \{x\in\p\Omega:\gamma(x)\cap Q_i\neq\emptyset\}$. Then $\sigma(\Delta_{Q_i})\leq C\ell(Q_i)^{n-1}$. Then, since $\dist(\cdot,\p\Omega)\simeq\ell(Q_i)$ on $Q_i$,
 \begin{align*}
  \|S(h)\|_{L^1(\p\Omega)}
  &\leq
  \sum_i \sigma(\Delta_{Q_i})
  \left(
   \int_{Q_i}|\nabla h|^2
   \dist(y,\p\Omega)^{2-n}\,dy
  \right)^{1/2}
  \\
  &\leq
  C\sum_i \ell(Q_i)^n\sup_{Q_i}|\nabla h|
  \leq
  C\sum_i\int_{2 Q_i}|\nabla h|\,dy
  \leq C\|\nabla h\|_{L^1(\Omega)}.
 \end{align*}
Finally, the mean-value property gives $|h(X_*)|\leq C\|h\|_{L^1(\Omega)}$. Combining these, the lemma is proved.
\end{proof}

\section{An interpolation lemma and solvability with $L^p$ data}
\label{sec-interpolation}

In this section, we upgrade the endpoint estimate with $H^1_{\cD}(\Omega)$ data to optimal $L^p$ estimates.
This is achieved by a real-valuable interpolation argument, originally due to Caffarelli--Peral \cite{CaffarelliPeral1998}. Here we use a version by Shen
\cite[Theorem~3.2]{Shen2007}. 

\begin{lemma}
 \label{lem:real-variable}
 Let $1<s<p<q<\infty$, and let $Q_0 \in \bR^n$ be a cube.  Suppose that
 $F\in L^1(2Q_0)$ is nonnegative and $h\in L^p(2Q_0)$.  Suppose that, for some $\beta < 1 < \alpha$ and every dyadic
 subcube $Q\subset Q_0$ with
 $\ell(Q)\leq \beta \ell(Q_0)$, there are nonnegative
 functions $F_Q$ and $R_Q$ on $2Q$ such that
 \begin{equation}
  F\leq F_Q+R_Q\qquad\text{on }2Q,
  \label{eq:RV-decomposition}
 \end{equation}
 \begin{equation}
  \fint_{2Q}F_Q\,dx
  \leq C_0\left(\fint_{ 2\alpha Q}|h|^s\,dx\right)^{1/s},
  \label{eq:RV-low}
 \end{equation}
 and
 \begin{equation}
  \left(\fint_{2Q}R_Q^q\,dx\right)^{1/q}
  \leq C_1\left\{
      \fint_{2\alpha Q}F\,dx
      +\left(\fint_{2\alpha Q}|h|^s\,dx\right)^{1/s}
  \right\}.
  \label{eq:RV-high}
 \end{equation}
 Then
 \begin{equation}
  \left(\fint_{Q_0}F^p\,dx\right)^{1/p}
  \leq C\left\{
       \fint_{2Q_0}F\,dx
       +\left(\fint_{2Q_0}|h|^p\,dx\right)^{1/p}
  \right\}.
  \label{eq:RV-conclusion}
 \end{equation}
\end{lemma}

Indeed, setting $H: = (\mathcal M (|h|^s))^{1/s}$, where $\mathcal{M}$ is the Hardy-Littlewood maximal operator restricted to $2Q_0$, the lemma follows directly from \cite[Theorem~3.2]{Shen2007}.

In this section, we mainly focus on $C^{1,\alpha}$ domains. In Section \ref{sec-260922-0511}, we discuss how to adapt the proof for Lipschitz polyhedrons.

\subsection{Reverse H\"older estimates}\label{sec-reverseholder}

We first give a optimal reverse H\"older estimate.
\begin{lemma} \label{lem-rev-holder}
    Let $s\in (1,4)$ and $\Omega$ be a $C^{1,\alpha}$ domain, $\alpha > 0$. 
    There exist $\theta_0>0$ and $r_0>0$, such that if $\Gamma$ is $\theta_0$-Reifenberg flat, then every solution $v \in W^{1,2}_{\cD}(2Q \cap \Omega)$ of
    \begin{equation} \label{eqn-260828-0335}
    -\Delta v = 0\,\,\text{in}\,\,2 Q \cap \Omega,
    \quad
    v = 0 \,\,\text{on}\,\, 2 Q \cap \cD,
    \quad
    \p v/\p \vec{\nu} = 0 \,\,\text{on}\,\, 2 Q \cap \cN
\end{equation}
satisfies
\begin{equation} \label{eqn-260826-1226}
    (\fint_{Q} |\nabla v|^s \chi_\Omega)^{1/s}
    \leq
    C \fint_{2Q} |\nabla v| \chi_\Omega,
    \quad \forall s\in (1,4)
\end{equation}
for cubes $Q$ centered in $\overline\Omega$ with $\ell(Q)<r_0$.
\end{lemma}

This follows from the optimal $W^{1,p}$ estimate in \cite[Corollary~4.2]{MR4232502}, which, more generally, proves for locally flat domains with locally flat interfaces.

\medskip

Using this, we further prove an optimal weighted reverse H\"older estimate.

\begin{lemma}
 \label{lem:weighted-reverse-Holder}
 Let $1<q<4/3$.  After decreasing the flatness constant depending on
$q$, $v$ satisfies
\begin{equation}
 \left(\fint_{Q}
  \left(\delta^{-1}|\nabla v|\chi_\Omega\right)^q dx
 \right)^{1/q}
 \leq C\fint_{2Q}
  \delta_\Gamma^{-1}|\nabla v|\chi_\Omega\,dx.
 \label{eq:weighted-reverse-Holder}
\end{equation}
\end{lemma}

\begin{proof}
We first prove \eqref{eq:weighted-reverse-Holder} when $\dist(2Q,\Gamma) \leq C \ell(Q)$ for some large $C>1$.  In this case, $\delta_\Gamma\leq C\ell(Q)$ on $2Q\cap\Omega$. Choose $2<r<4$ sufficiently close to $4$ so that
\begin{equation}
 \beta:=\frac{qr}{r-q}<2.
 \label{eq:weighted-beta}
\end{equation}
Next choose $0<\epsi <2-\beta$ and decrease the flatness constant so that
Lemma \ref{lem-codim2-vol} applies.

By H\"older's inequality,
\begin{equation} \label{eqn-260826-1216}
 \left(\fint_{Q}
  \left( \delta^{-1} |\nabla v|\chi_\Omega\right)^q dx
 \right)^{1/q}
 \leq
 \left(\fint_{Q}|\nabla v|^r\chi_\Omega\,dx\right)^{1/r}
 \left(\fint_{Q}\delta^{-\beta}\chi_\Omega\,dx\right)^{1/\beta}.
\end{equation}

From Lemma \ref{lem-rev-holder} and Lemma \ref{lem-codim2-vol},
\begin{align*}
    \text{RHS of}\,\,\eqref{eqn-260826-1216}
     \leq C\ell(Q)^{-1}
       \fint_{2Q}|\nabla v|\chi_\Omega\,dx
 \leq C\fint_{2Q}
       \delta^{-1}|\nabla v|\chi_\Omega\,dx.
\end{align*}
When $\dist(2Q,\Gamma)> C \ell(Q)$, the desired estimate follows directly from Lemma \ref{lem-rev-holder} with $s=q$, since $\delta \geq c \ell(Q)$ on $2Q$ for some $c>0$. 
\end{proof}

\subsection{Proof of Theorem \ref{thm:poisson-smooth}}

Choose
\begin{equation}
       1<s<p<q< 4/3.
       \label{eq:interpolation-exponents}
\end{equation}
Decrease the flatness constant so that Lemma~\ref{lem:weighted-reverse-Holder} is available with exponent $q$. 
By approximation, we may assume that $f\in L^2(\Omega)$, from which an energy solution $u \in W^{1,2}_\cD(\Omega)$ is available. 
Since $f\in H^1_\cD(\Omega)$, we have $u \in W^{2,1}(\Omega) \cap W^{1,1}_\cD(\Omega)$ and
\[
 F:=\delta^{-1}|\nabla u|\chi_\Omega
 \in L^1(\bR^n).
\]

Fix a sufficiently small cube $Q$ in a coordinate patch, $Q\cap \overline \Omega \neq \emptyset$. We first construct the localized datum used to define $w_Q$. 
Let $B_Q$ be a ball containing $8Q$, with radius comparable to $\ell (Q)$.

If $2B_Q\cap\cD\neq\emptyset$, no cancellation is required for an atom associated with $B_Q$. We set
\[
 g_Q:=f\chi_{4Q\cap\Omega}.
\]

If $2B_Q\cap\cD=\emptyset$, $\int_\Omega g_Q=0$ is needed. Let
$E_Q := (8Q\setminus 4Q)\cap\Omega$. Then, $|E_Q|\simeq|Q|$. Set
\[
 g_Q:=f\chi_{4Q\cap\Omega}
 -|E_Q|^{-1} \int_{4Q\cap\Omega}f\,dx
  \chi_{E_Q}.
\]

By H\"older's inequality, $\|g_Q\|_{L^s(\Omega)} \leq C\|f\|_{L^s(8Q\cap\Omega)}$.
After normalization, $g_Q$ is an adapted $(1,s)_\cD$-atom. Lemma
\ref{lem:atom-decomposition} yields
\begin{equation}
 \|g_Q\|_{H^1_\cD(\Omega)}
 \leq C|Q|
 \left(\fint_{8Q}|f\chi_\Omega|^s\,dx\right)^{1/s}.
 \label{eq:localized-datum-bound}
\end{equation}

Let $w_Q$ be the solution with datum $g_Q$, obtained from case~(a), and let
\[
 v_Q:=u-w_Q,
 \qquad
 F_Q:= \delta^{-1}|\nabla w_Q|\chi_\Omega,
 \qquad
 R_Q:=\delta^{-1}|\nabla v_Q|\chi_\Omega.
\]
Clearly, 
\begin{equation}
       F\leq F_Q+R_Q.
       \label{eq:weighted-decomposition}
\end{equation}

Proposition \ref{prop:atom-data} and \eqref{eq:localized-datum-bound} give
\begin{equation}
\int_\Omega F_Q\,dx
 \leq \|g_Q\|_{H_{\cD}^1(\Omega)}
 \leq C|Q| \left(\fint_{8Q}|f\chi_\Omega|^s\,dx\right)^{1/s}.
 \label{eq:interpolation-low}
\end{equation}

Since $g_Q=f$ in $4Q\cap\Omega$, $v_Q$ satisfies a homogeneous problem here.
By Lemma~\ref{lem:weighted-reverse-Holder},
\begin{align*}
 \left(\fint_{2Q}R_Q^q\,dx\right)^{1/q}
 \leq C\fint_{4Q}R_Q\,dx.
\end{align*}
Noting $R_Q \leq F + F_Q$, from this and \eqref{eq:interpolation-low}, we obtain
\begin{equation}
 \left(\fint_{2Q}R_Q^q\,dx\right)^{1/q}
 \leq C\left\{
      \fint_{4 Q}F\,dx
      +\left(\fint_{8 Q}|f\chi_\Omega|^s\,dx\right)^{1/s}
 \right\}.
 \label{eq:interpolation-high}
\end{equation}

Equations \eqref{eq:weighted-decomposition},
\eqref{eq:interpolation-low}, and \eqref{eq:interpolation-high} verify
the hypotheses of Lemma~\ref{lem:real-variable}.  Applying the same lemma in
each coordinate patch and then using a finite covering of
$\overline\Omega$, we obtain
\begin{equation}
 \left\| \delta^{-1} \nabla u\right\|_{L^p(\Omega)}
 \leq C\left(
     \left\|\delta^{-1}\nabla u\right\|_{L^1(\Omega)}
     +\|f\|_{L^p(\Omega)}
 \right).
 \label{eq:weighted-Lp-intermediate}
\end{equation}
Finally, the embedding $L^p(\Omega) \hookrightarrow H^1_\cD(\Omega)$ and the endpoint estimate in case (a) show that the first term on the right is bounded
by $C_p\|f\|_{L^p(\Omega)}$. 
Combining this with Proposition~\ref{prop:weighted-Whitney} proves
the desired estimate.

\subsection{Proof of Theorem \ref{thm:poisson-polyhedron}}\label{sec-260922-0511}
The proof follows the same strategy. Below we give a sketch.

From the endpoint estimate and solvability in Hardy spaces in Section \ref{sec-hardy}, we apply the interpolation. The only difference is, instead of a $W^{1,4-\epsilon}$ estimate for solution $v$ to homogeneous problems \eqref{eqn-260828-0335}, which requires the domain to be locally flat, here we apply a $W^{1,s}$ estimate for $s$ slightly bigger than $2$, which works for any Lipschitz domain with $(\cD, \cN)$ satisfying a corkscrew condition as in Lemma \ref{lem-260819-1105}. 
More precisely, there exists some $\epsi_0$ depending only on the Lipschitz character of $\Omega$ and the constant in the corkscrew condition, such that if $v \in W^{1,2}_{\cD} (2Q\cap \Omega)$ satisfies a homogeneous mixed problem, then $v \in W^{1,s}_{\cD} (2Q\cap \Omega)$ for all $s < 2+\epsi_0$, with
\begin{equation*}
    (\fint_{Q} |\nabla v|^s \chi_{\Omega})^{1/s}
    \leq
    C \fint_{2Q} |\nabla v| \chi_{\Omega}.
\end{equation*}

This in turn leads to the weighted estimate in Lemma \ref{lem:weighted-reverse-Holder} with $1<q<1 + \epsi_0/(4+\epsi_0)$. Indeed, noting that we have $\delta^{-\beta} \in L^1_{loc}$ for any $\beta < 2$ on a Lipschitz polyhedron, the restriction \eqref{eq:weighted-beta} leads to $q^{-1} > 1/2 + 1/(2+\epsi_0)$, i.e., $q < 1 + \epsi_0/(4+\epsi_0)$.
From this, the interpolation argument leads to the desired estimate with $p < q < 1 + \epsi_0 (4+\epsi_0)^{-1}$, and the solvability follows.

\subsection{Inhomogeneous boundary data}
\label{sec:inhomogeneous-data}

We prove Corollaries~\ref{cor-inhomo-1}
and~\ref{cor-inhomo-2} by extending the boundary data.
It suffices to construct $G\in W^{2,p}(\Omega)$ such that
\begin{equation}
 G|_\cD=g_\cD,\qquad
 \partial_\nu G|_\cN=g_\cN,\qquad
 \|G\|_{W^{2,p}(\Omega)}\leq C \big( \|g_\cD\|_{W^{2-1/p,p}(\cD)} + \|g_\cN\|_{W^{1-1/p,p}(\cN)} \big).
 \label{eq:boundary-lifting}
\end{equation}
For a polyhedron $\Omega$, replace the boundary norms by the corresponding sums over faces.

For Corollary~\ref{cor-inhomo-1}, extend $g_\cD$ and $g_\cN$ to $\partial\Omega$, using the restriction-space interpretation of their norms.
Since $\alpha>1-1/p$, the inverse trace theorem \cite[Proposition~7.12]{MazyaMitreaShaposhnikova2010} gives an extension satisfying \eqref{eq:boundary-lifting}.

For Corollary~\ref{cor-inhomo-2}, choose $p_1=\min\{p_0,2\}$ when $n=2$, and $p_1=\min\{p_0,3/2\}$ when $n\geq3$. We apply the general inverse trace theorem in \cite[Proposition~7.8]{MazyaMitreaShaposhnikova2010}.
On a polyhedron, the boundary is only piecewise smooth, and the unit normal jumps across intersections of adjacent faces. Consequently, facewise regularity must be supplemented by compatibility conditions for the boundary traces. 

For Neumann data, simply extend them by zero to the Dirichlet faces. Since $p<2$, we have $1-1/p<1/p$ and therefore no compatibility condition on common edges is needed. In particular, denoting the resulting function on $\p\Omega$ by $b$, we have $\nu b \in W^{1-1/p,p}(\partial \Omega)$.
To extend Dirichlet data to Neumann faces, when $n=2$, simply interpolate linearly from the values at the vertices. When $n\geq3$, extend the boundary values on each Neumann face using the traces inherited from adjacent Dirichlet edges and zero on the remaining codimension-two edges. These traces have regularity $W^{2-2/p,p}$.
The restriction $p<3/2$ gives $2-2/p<1/p$, so they can be extended without further compatibility at intersections of codimension at least three. The resulting function, denoted by $a$, satisfies $a\in W^{2-1/p,p}(\p\Omega)$.

With these, we have $\nabla_{\tan}a+\nu b
\in W^{1-1/p,p}(\p\Omega)$, and the desired extension follows form \cite[Theorem~7.8]{MazyaMitreaShaposhnikova2010}.

\section{Discussion on Lam\'e systems}
\label{sec:lame}

The weak solvability and regularity theory for Lam\'e systems have long been studied; see, for example, \cite[Section~6.3]{MR936420} and the references therein. For pure Dirichlet or pure traction (conormal) problems, if domains, coefficients, and data are all smooth, corresponding solutions are also smooth by the classical Agmon--Douglis--Nirenberg theory for elliptic systems with complementing boundary conditions in \cite{MR162050}. 

Just like scalar equations, singularities for Lam\'e systems may occur on nonsmooth domains or at interfaces between two types of boundary conditions. In dimensions two and three, the spectral analysis and singular behaviors of solutions to Lam\'e and general elliptic systems near corners, edges, and vertices have been developed extensively; see, for instance, \cite{MR996909,MR1006837,MR1205404,MR1147281}.
Along this direction, regularity or well-posedness in weighted Sobolev spaces on polyhedral domains have been developed under suitable assumptions; see \cite{MR2564468} and the references therein. Let us also mention the work by Ott and Brown \cite{MR3040944} on general planar Lipschitz domains, in which solvability with control of the non-tangential maximal function of the gradient were obtained.

In this section, we explain how the preceding argument applies to the Lam\'e systems. 
Let $\Omega\subset\bR^n$ be a bounded Lipschitz polygon with a facewise decomposition $(\cD,\cN)$, where $\cD\neq\emptyset$. For a vector field $u:\Omega\to\bR^n$, set
\[
 \epsi(u):=\frac12\bigl(\nabla u+(\nabla u)^\top\bigr),
 \qquad
 \sigma(u):=2\mu\varepsilon(u)+\lambda\operatorname{div}(u)I_n,
\]
where $I_n$ is the identity matrix and $\lambda, \mu$ are parameters satisfying
\begin{equation} \label{eqn-lame-coeff}
    \mu > 0,
    \quad
    \lambda + \frac{2}{n} \mu > 0.
\end{equation}
We consider the mixed displacement-traction problem
\begin{equation}
 \left\{
 \begin{aligned}
  -\operatorname{div}\sigma(u)&=f &&\text{in }\Omega,\\
  u&=0 &&\text{on }\cD,\\
  \sigma(u)\nu&=0 &&\text{on }\cN.
 \end{aligned}
 \right.
 \tag{$\mathrm L$}
 \label{prob:lame}
\end{equation}

A vector field $u\in W^{1,p}_{\cD}(\Omega;\bR^n)$ is a weak solution of
\eqref{prob:lame} if
\begin{equation}
 \int_\Omega \sigma(u):\varepsilon(\varphi)\,dx
 =\int_\Omega f\cdot\varphi\,dx,
 \qquad
 \forall\varphi\in [C^\infty_{\cD}(\Omega)]^2.
 \label{eq:lame-weak}
\end{equation}

With straightforward modifications, our methods for scalar equations also work for Lam\'e systems. The only missing ingredient is the boundary H\"older estimate. At dimension two ($n=2$), this can be resolved by a $W^{1,2+\epsi}$ estimate and the Sobolev embedding. Discussions for extensions to higher dimensions will be made at the end.

\begin{theorem}
 \label{thm:lame-polygon}
 Let $\Omega\subset\bR^2$ be a bounded and Lipschitz polygon with a facewise decomposition $(\cD,\cN)$, where $\cD\neq\emptyset$.
 There exists $p_1 > 1$ depending on $\Omega$ and the Lam\'e coefficients, such that the following hold.
 \begin{enumerate}
  \item If $f\in H^1_{\cD}(\Omega;\bR^2)$, there exists a unique solution $u\in W^{2,1}(\Omega;\bR^2)\cap W^{1,1}_{\cD}(\Omega;\bR^2)$ of \eqref{prob:lame}, and
  \begin{equation}
   \|D^2u\|_{L^1(\Omega)}
   +\|\delta_\Sigma^{-1}\nabla u\|_{L^1(\Omega)}
   \leq C\|f\|_{H^1_{\cD}(\Omega;\bR^2)}.
   \label{eq:lame-hardy-estimate}
  \end{equation}
  \item If $1<p<p_1$ and $f\in L^p(\Omega;\bR^2)$, there exists a unique
  solution $u\in W^{2,p}(\Omega;\bR^2)\cap W^{1,p}_{\cD}(\Omega;\bR^2)$, and
  \begin{equation}
   \|D^2u\|_{L^p(\Omega)}
   +\|\delta_\Sigma^{-1}\nabla u\|_{L^p(\Omega)}
   \leq C_p\|f\|_{L^p(\Omega;\bR^2)}.
   \label{eq:lame-Lp-estimate}
  \end{equation}
 \end{enumerate}
\end{theorem}
Here, the vector-valued atoms and $H^1_\cD(\Omega ; \bR^2)$ are defined componentwise.

\begin{proof}[Sketch of the proof]
To start with, under \eqref{eqn-lame-coeff}, Korn's inequality gives coercivity of the elasticity form on $W^{1,2}_{\cD}(\Omega;\bR^2)$ and hence the estimate in energy space $W^{1,2}_\cD$. The Caccioppoli inequality, the Sobolev--Poincar\'e inequality, and Gehring's lemma yield an exponent $s>2$ such that homogeneous mixed solutions satisfy
\begin{equation}
 \left(\fint_{B_r(x)\cap\Omega}|\nabla v|^s\,dx\right)^{1/s}
 \leq C\left(\fint_{B_{2r}(x)\cap\Omega}|\nabla v|^2\,dx\right)^{1/2}
 \label{eq:lame-meyers}
\end{equation}
on sufficiently small balls where the homogeneous problem holds in $B_{2r}(x)\cap\Omega$; see \cite[Theorem~3.1]{MR3040944}. 

Similarly, the inhomogeneous estimate with $f \in L^p(\Omega;\bR^2)$, $p > 2n/(n+2) = 1$, follows by the same duality argument in Lemma~\ref{lem-260819-1105}. 

The same proof of the key Whitney reduction in Proposition~\ref{prop:weighted-Whitney} still works for Lam\'e systems, noting that away from $\Sigma$, the local $W^{2,q}$ estimates are interior or half-space estimates for pure displacement or traction problems, which can be found in \cite{MR162050}. 
The only additional ingredient in the atom argument is a boundary H\"older estimate. In dimension two, \eqref{eq:lame-meyers} and the Sobolev--Morrey embedding give, with $\gamma=1-2/s>0$,
\begin{equation}
 r^\gamma[v]_{C^\gamma(\overline{\Omega}\cap B_r(x))}
 \leq Cr\left(\fint_{B_{2r}(x)\cap\Omega}|\nabla v|^2\,dx\right)^{1/2}.
 \label{eq:lame-holder}
\end{equation}
Using this estimate for the adjoint solution in Step~2 of the proof of Proposition~\ref{prop:atom-data}, together with the cancellation of the vector-valued atom or the zero Dirichlet trace, gives the same annular decay. The remaining steps of that proof and summation of an atomic decomposition give~\eqref{eq:lame-hardy-estimate} and endpoint existence.

The interpolation argument in Section~\ref{sec-260922-0511} then applies to vector-valued data. Indeed, \eqref{eq:lame-meyers} gives the required weighted reverse H\"older estimate for $1<q<2s/(s+2)$, since $qs/(s-q)<2$. Lemma~\ref{lem:real-variable} and the Whitney reduction therefore yield~\eqref{eq:lame-Lp-estimate} for $1<p<p_1:=2s/(s+2)$.

Finally, $W^{2,1}(\Omega;\bR^2)\hookrightarrow W^{1,2}(\Omega;\bR^2)$. A solution of the homogeneous problem in the endpoint class therefore lies in the energy space. Testing by $u$ and applying Korn's inequality proves uniqueness.
\end{proof}

\subsection*{Inhomogeneous data}
As in the scalar case, extension of boundary data gives the solvability of inhomogeneous problem
\begin{equation}
 \left\{
 \begin{aligned}
  -\operatorname{div}\sigma(u)&=f &&\text{in }\Omega,\\
  u&=g_\cD &&\text{on }\cD,\\
  \sigma(u)\nu&=g_\cN &&\text{on }\cN
 \end{aligned}
 \right.
  \tag{$\mathcal{L}$}
 \label{prob:lameinhomogeous}
\end{equation}

\begin{corollary}\label{cor-inhomo-lame}
Suppose that the assumptions of Theorem~\ref{thm:lame-polygon} hold, and let $1<p<p_1$.
Let $f\in L^p(\Omega;\bR^2)$, and suppose that
\[
 g_{\cD,j}:=g_\cD|_{F_j}
 \in W^{2-1/p,p}(F_j;\bR^2)
 \quad\text{for every }F_j\subset\cD,
\]
and
\[
 g_{\cN,j}:=g_\cN|_{F_j}
 \in W^{1-1/p,p}(F_j;\bR^2)
 \quad\text{for every }F_j\subset\cN.
\]
Assume that the Dirichlet data agree at every common vertex of two Dirichlet edges.
Then there exists a unique solution $u\in W^{2,p}(\Omega;\bR^2)$ of \eqref{prob:lameinhomogeous} satisfying
\[
 \|u\|_{W^{2,p}(\Omega)}
 \leq C\left(
  \|f\|_{L^p(\Omega)}
  +\sum_{F_j\subset\cD}
    \|g_{\cD,j}\|_{W^{2-1/p,p}(F_j)}
  +\sum_{F_j\subset\cN}
    \|g_{\cN,j}\|_{W^{1-1/p,p}(F_j)}
 \right).
\]

\end{corollary}

\begin{proof}[Sketch of the proof]
We modify the scalar extension argument for Corollary~\ref{cor-inhomo-2} in dimension two.
Extend the displacement data to a function $a$ on $\p\Omega$, componentwise as in that argument.
The difference lies in determining the normal-derivative data $b$ required by the inverse trace theorem from the prescribed traction data and the extended Dirichlet data.

On each Neumann edge $F_j$, let $\tau_j$ and $\nu_j$ be a unit tangent and the outward unit normal, and write $a'_j=\partial_{\tau_j}a$.
Prescribe $b_j$ by
\[
 b_j\cdot\tau_j
 =\frac{g_{\cN,j}\cdot\tau_j}{\mu}
  -a'_j\cdot\nu_j,
 \qquad
 b_j\cdot\nu_j
 =\frac{g_{\cN,j}\cdot\nu_j
        -\lambda a'_j\cdot\tau_j}
       {\lambda+2\mu},
\]
and set $b_j=0$ on the Dirichlet edges.
Since $\mu>0$ and $\lambda+2\mu>0$, these formulas define edgewise $W^{1-1/p,p}$ data with the required bound. They ensure $\sigma(G)\nu=g_\cN$ whenever $G=a$ and $\partial_\nu G=b$ on the boundary. The rest of the proof remain unchanged.
\end{proof}

\subsection*{Discussion for higher dimensions}

When $n\geq 3$, the main missing ingredient for extending the above argument is a local boundary H\"older estimate for homogeneous solutions. 
In dimension three, such estimates are available for the pure Dirichlet and pure traction problems on general Lipschitz domains; see, for instance, \cite{DahlbergKenigVerchota1988,DahlbergKenig1990}. The special role of dimension three can be understood through non-tangential maximal-function estimates; see \cite{Shen2003Morrey,Shen2006LpDirichlet}.
Indeed, the localized pure-boundary theory controls the non-tangential maximal function of $\nabla v$ in $L^s$ for some $s>2$. When $n=3$, this exponent exceeds the boundary dimension $n-1=2$.
For $n\geq4$, such range does not reach an exponent above $n-1$.
In fact, for $n\geq4$, even the weak maximum principle for the Dirichlet Lam\'e problem on an arbitrary Lipschitz domain remains open; see, for example, \cite[Problem~3.2.38]{MR1282720} and \cite{Shen2006LpDirichlet, Zhuge2020WeakMaximum}.

For mixed boundary conditions, the singularity along the interface $\Gamma$ creates an additional difficulty. For instance, the expected optimal range for non-tangential maximal function becomes $s<4/3$, which is below $n-1$ when $n=3$. This leads naturally to the following question:
\begin{center}
    \textit{Does a scale-invariant local boundary H\"older estimate hold for homogeneous Lam\'e systems with mixed Dirichlet--traction boundary conditions on three dimensional Lipschitz domains?}
\end{center}

To the best of our knowledge, such a H\"older estimate is not available even for arbitrary three-dimensional Lipschitz polyhedra. Nicaise \cite{MR1205404} proved the stronger regularity $H^{3/2+\varepsilon}$, and hence H\"older continuity, under additional geometric assumptions. In particular, the interior angle at every Dirichlet--traction edge is required to be less than $\pi$, together with an additional condition at the vertices. Nicaise further conjectured that the latter vertex condition could be removed \cite[Remark~4.4]{MR1205404}.
An affirmative answer, even restricted to Lipschitz polyhedra, would provide the decay estimate needed for the atomic data estimate and the corresponding $W^{2,p}$ solvability for $p$ close to $1$.

\section*{Acknowlegement}
Generative AI tools were used to assist with language editing and improving the presentation of the manuscript. The authors take full responsibility for the content of the manuscript.

\newcommand{\cprime}{\'}
\bibliographystyle{plain}

\bibliography{biblio.bib}

\end{document}